\documentclass[11pt]{article}
\usepackage{graphics}
\usepackage{amssymb}

\usepackage{mathrsfs}
\usepackage{makeidx}
\usepackage{color}
\usepackage{graphicx}
\usepackage{subfigure}
\usepackage{multicol}
\usepackage{multirow}
\usepackage{float}
\usepackage{booktabs}
\usepackage{epsfig}

\usepackage{epstopdf}
\usepackage{caption}

\usepackage[colorlinks,
            linkcolor=red,
            anchorcolor=blue,
            citecolor=blue
            ]{hyperref}
\usepackage{latexsym, cite, bm, amsthm, amsmath}
\usepackage{cleveref}

\usepackage{algorithm,algorithmic}

\usepackage{enumerate}
\usepackage{marginnote}

\usepackage{pifont}

\def \[{\begin{equation}}
	\def \]{\end{equation}}

\newtheorem{theorem}{Theorem}[section]
\newtheorem{corollary}{Corollary}[section]
\newtheorem{lemma}{Lemma}[section]

\newtheorem{definition}{Definition}[section]

\newtheorem{remark}{Remark}[section]

\newtheorem{example}{Example}[section]
\numberwithin{equation}{section}

\title{An inverse-free fixed-time stable dynamical system and its forward-Euler discretization for solving generalized absolute value equations}

\usepackage[marginal]{footmisc}
\usepackage{authblk}
\author[a]{Xuehua Li\thanks{Email address: 3222714384@qq.com.}}
\author[b]{Linjie Chen\thanks{Supported partially by the Fujian Social Science Foundation (No. FJ2024BF057). Email address: clj@fjnu.edu.cn.}}
\author[c]{Dongmei Yu\thanks{Supported partially by the National Natural Science Foundation of China (No. 12201275), the Natural Science Foundation of Liaoning Province (No. 2024-MS-206) and the Liaoning Provincial Department of Education (Nos. JYTZD2023072, LJ112410147046, LJ242410147027). Email: yudongmei1113@163.com.}}
\author[b]{Cairong Chen\thanks{Corresponding author. Supported partially by the Natural Science Foundation of Fujian Province (No. 2025J01673) and the Fujian Alliance of Mathematics (No. 2023SXLMQN03). Email address: cairongchen@fjnu.edu.cn.}}
\author[d]{Deren Han\thanks{Supported partially by the National Natural Science Foundation of China (No. 12131004) and the Ministry of Science and Technology of China (No. 2021YFA1003600). Email address: handr@buaa.edu.cn.}}
\affil[a]{School of Mathematics and Computational Science, Xiangtan University, Xiangtan, 411100, P.R. China}
\affil[b]{School of Mathematics and Statistics, Fujian Normal University, Fuzhou, 350117, P.R. China}
\affil[c]{Institute for Optimization and Decision Analytics, Liaoning Technical University, Fuxin, 123000, P.R. China.}
\affil[d]{School of Mathematical Sciences, Beihang University, Beijing, 100191, P.R. China}

\begin{document}
\date{\today}
\maketitle

\begin{abstract}
An inverse-free dynamical system is proposed to solve the generalized absolute value equation (GAVE) with a fixed time convergence, where the time of convergence is finite and is uniformly bounded for all initial points. Moreover, an iterative method obtained by using the forward-Euler discretization of the proposed dynamic model is developed and sufficient conditions which guarantee that the discrete iteration globally converge
to an arbitrarily small neighborhood of the unique solution of GAVE within a finite number of iterative steps are given. Numerical results illustrate the effectiveness of the proposed methods.
\end{abstract}

\section{Introduction}\label{sec:intro}

Consider the generalized absolute value equation~(GAVE)
\begin{equation}\label{eq:gave}
  Ax - B|x|=c,
\end{equation}
where~$A,B\in\mathbb{R}^{m\times n}$ and $c\in\mathbb{R}^m$ are known, and $x\in\mathbb{R}^n$ is the unknown vector. Here, $|x|=[|x_1|,|x_2|,\ldots,|x_n|]^\top$. If $B$ is invertible, then GAVE~\eqref{eq:gave} can turn into
\begin{equation}\label{eq:ave}
Ax-|x|=c,
\end{equation}
which is the so-called absolute value equation (AVE). Due to the existence of the absolute value term $|x|$, solving GAVE~\eqref{eq:gave} is generally NP-hard~\cite{mang2007a}.

The linear complementarity problem (LCP) \cite{cops1992} is a well-known problem in mathematical programming. Recall that LCP is to find a $z\in \mathbb{R}^\ell$ such that
\begin{equation}\label{eq:lcp}
w = Mz + q\ge 0, \quad z\ge 0,\quad w^\top z = 0,
\end{equation}
where $M\in \mathbb{R}^{\ell\times \ell}$ and $q\in \mathbb{R}^\ell$ are given. LCP~\eqref{eq:lcp} is a special case of the following horizontal linear complementarity problem (HLCP): find two vectors $z,w\in \mathbb{R}^{\ell}$  such that
\begin{equation}\label{eq:hlcp}
P z - Qw = p, \quad z\ge 0,\quad w\ge 0, \quad w^\top z = 0,
\end{equation}
where $P,Q\in \mathbb{R}^{\ell\times \ell}$ and $p\in \mathbb{R}^{\ell}$ are known. In addition, LCP~\eqref{eq:lcp} is also a special case of the following generalized linear complementarity problem (GLCP): find an  $x\in \mathbb{R}^{\ell}$ such that
\begin{equation*}\label{eq:glcp}
Ex + e \ge 0, \quad Fx + f\ge 0,\quad (Ex+e)^\top (Fx + f)= 0,
\end{equation*}
where $E,F\in \mathbb{R}^{\ell \times \ell}$ and $e, f\in \mathbb{R}^{\ell}$ are known.

There is a strong connection between GAVE~\eqref{eq:gave} (or AVE~\eqref{eq:ave}) and the LCP as well as its extensions. In \cite{mame2006}, AVE~\eqref{eq:ave} is equivalently reformulated as the following GLCP:  find an $x\in\mathbb{R}^n$ such that
\begin{equation}\label{eq:ave2glcp}
Ax + x -c\ge 0, \quad Ax-x -c\ge 0, \quad (Ax + x -c)^\top ( Ax-x -c) = 0,
\end{equation}
which, under the assumption that $1$ is not an eigenvalue of $A$, can be reduced to the following LCP: find a $z\in \mathbb{R}^n$ (which then solves AVE~\eqref{eq:ave} by $x = (A-I)^{-1} (z+c)$) such that
\begin{equation}\label{eq:ave2lcp}
w = (A+I)(A-I)^{-1}z + q \ge 0, \quad z\ge 0, \quad w^\top z = 0
\end{equation}
with
\begin{equation}\label{eq:qz}
q = [(A+I)(A-I)^{-1} - I]c.
\end{equation}
In \cite{prok2009}, GAVE~\eqref{eq:gave} is reformulated as a standard LCP without any additional assumption on the coefficient matrices $A$ and $B$. However, the dimension of the matrix in the obtained LCP is greater than that of $A$ (and $B$) in the original GAVE~\eqref{eq:gave}. In \cite{huhu2010}, based on \eqref{eq:ave2glcp}, AVE~\eqref{eq:ave} is transformed into the following HLCP without any assumption on the coefficient matrix $A$: find $w\in \mathbb{R}^n$ and $z\in \mathbb{R}^n$ (which then solve AVE~\eqref{eq:ave} by $x = \frac{1}{2}(z -w)$) such that
\begin{equation*}\label{ave2hlcp}
(I+A) w - (A - I) z = -2c,\quad w\ge 0,\quad z\ge 0, \quad w^\top z = 0,
\end{equation*}
which is reduced to the following LCP: find a $z\in \mathbb{R}^n$ such that
\begin{equation*}\label{eq:ave2lcpn}
w = (AD - I)^{-1}(AD + I) z + 2(AD - I)^{-1}c\ge 0, \quad z\ge 0, \quad w^\top z = 0,
\end{equation*}
where $D$ is a diagonal matrix with its diagonal elements being $1$ or $-1$, which is determined by an index set; see the proof of  \cite[Lemma~2.1]{huhu2010} for more details. Conversely, letting $w = |x| + x$ and $z = |x| - x$,  we can find a solution to LCP~\eqref{eq:lcp} by solving the following GAVE:
\begin{equation}\label{eq:lcp2gave}
(M+I)x - (M-I)|x| = q,
\end{equation}
which can be transformed into AVE
\begin{equation}\label{eq:lcp2ave}
(M-I)^{-1}(M+I)x - |x| = (M-I)^{-1}q
\end{equation}
whenever $1$ is not an eigenvalue of $M$. Without loss of generality, we can always assume that $1$ is not an eigenvalue of $M$. Then, the solution of LCP can be obtained by solving GAVE~\eqref{eq:lcp2gave} or AVE~\eqref{eq:lcp2ave}. However, since $(M-I)^{-1}$ does not be required, from a computational perspective, GAVE~\eqref{eq:lcp2gave} is more attractive than AVE~\eqref{eq:lcp2ave}.
In \cite{mezz2020}, the equivalence between GAVE~\eqref{eq:gave} and HLCP~\eqref{eq:hlcp} is investigated. Specifically,  if $x$ is a solution of GAVE~\eqref{eq:gave}, then the vectors $z = \max\{x,0\}$ and $w = \max\{-x,0\}$ solve the HLCP
\begin{equation}\label{eq:gave2hlcp}
(A-B)z - (A+B)w = c, \quad z\ge 0,\quad w\ge 0, \quad w^\top z = 0.
\end{equation}
Conversely, if $(z,w)$ solve HLCP~\eqref{eq:hlcp}, then $x = z-w$ solves GAVE~\eqref{eq:gave} with $A = \frac{1}{2}(P +Q)$, $B = \frac{1}{2}(Q - P)$ and $c = p$.

Since AVE~\eqref{eq:ave} can be transformed into the standard LCP, GLCP or HLCP, we can find a solution to AVE~\eqref{eq:ave} by solving LCP, GLCP or HLCP. Based on HLCP~\eqref{eq:gave2hlcp}, Gao and Wang \cite{gawa2014} propose a one-layer neural network for solving AVE~\eqref{eq:ave} and prove that the neural network is globally exponentially stable if $\sigma_{\min}(A) > 1$. Based on LCP~\eqref{eq:ave2lcp}--\eqref{eq:qz}, Huang and Cui \cite{hucu2017}, Mansoori et al. \cite{maee2017} and Mansoori and Erfanian \cite{maer2018} propose three neural networks for solving AVE~\eqref{eq:ave} which are proved to be globally asymptotically stable under certain conditions. We should mention that the stability condition of the neural network proposed in \cite{hucu2017} is corrected in \cite{yuch2023}. Ju et al. \cite{jlhh2022} propose a novel projection neurodynamic network with fixed-time convergence for solving AVE~\eqref{eq:ave}. During the construction of all neural networks mentioned above, a matrix inversion is required. In order to overcome this drawback, based on~\eqref{eq:ave2glcp}, Chen et al. \cite{cyyh2021} propose an inverse-free dynamical system for solving AVE~\eqref{eq:ave} and the (globally) asymptotical stability is proved. Yu et al. \cite{ycyh2023} propose an inertial inverse-free dynamical system for solving AVE~\eqref{eq:ave} and the asymptotical convergence is proved. Li et al. \cite{lyyh2023} propose a new fixed-time dynamical system for solving AVE~\eqref{eq:ave}. In \cite{jyfc2023}, compared with \cite{lyyh2023}, two more accurate upper bounds of settle time and robustness analysis are given. Zhang et al. \cite{zlzh2024} propose two new accelerated fixed-time stable dynamic systems for solving AVE~\eqref{eq:ave}. Yu et al. \cite{yzch2024} propose two inverse-free neural network models with delays for solving AVE~\eqref{eq:ave}.

LCP has a wide range of applications in applied science and technology \cite{cops1992}. As mentioned earlier,  LCP can be solved by solving a GAVE or AVE. This is exactly the idea of the modulus-based methods \cite{bai2010,murt1988,bokh1980,zhyi2013,doji2009,hatz2009}, to name only a few.
Though there are many dynamical systems for solving AVE~\eqref{eq:ave}, the dynamical system for solving GAVE~\eqref{eq:gave} (when $B$ is singular, GAVE~\eqref{eq:gave} cannot be reformulated as AVE~\eqref{eq:ave}) is rare. According to \cite{prok2009}, constructing a dynamical system for solving GAVE~\eqref{eq:gave} through LCP is not wise due to the expansion of the dimension. When $A$ is nonsingular, the neural network proposed in \cite{gawa2014} can be adopted to solve  GAVE~\eqref{eq:gave}, which is described as follows:

\begin{itemize}
  \item state equation
  \begin{equation}\label{eq:nn4gave}
  \frac{{\rm d} z}{{\rm d}t} = \frac{1}{2}\rho (|A^{-1} (Bz + c)| - z),
  \end{equation}

  \item output equation
  \begin{equation}\label{eq:outx}
  x = A^{-1} (Bz + c),
  \end{equation}
\end{itemize}
where $\rho > 0$ is a scaling constant. Obviously, a matrix inversion is required in \eqref{eq:nn4gave} and~\eqref{eq:outx}. According to \cite[Theorem~5]{gawa2014}, we can prove that the neural network \eqref{eq:nn4gave} is globally exponentially stable if $\sigma_{\min}(A) > \sigma_{\max}(B)$. Globally exponential or asymptotical stability characterizes the property of the equilibrium point as time goes to infinity, which seems hard to be controlled in real-world. To overcome this drawback, the control theory provides many systems that exhibit finite-time convergence to the equilibrium, especially the fixed-time stability, see, e.g, \cite{poly2012}. The goal of this paper is to construct an inverse-free and globally fixed-time stable dynamical system for solving GAVE~\eqref{eq:gave}. We are interested in GAVE~\eqref{eq:gave} not only due to its connections with complementarity problems, but also due to its applications in  the linear interval equations  \cite{rohn1989,liwu2025},   the cancellable biometric system \cite{dnhk2023} and the ridge regression problem \cite{xiqh2024}.

Our work here is inspired by \cite{lyyh2023}. The theoretical analysis in~\cite{lyyh2023} relies on \cite[Theorem~3.5]{cyyh2021}  and~\cite[Theorem~4.1]{cyyh2021}. The proofs of \cite[Theorem~3.5]{cyyh2021}  and~\cite[Theorem~4.1]{cyyh2021} depend on the following two key properties:
\begin{itemize}
  \item the equivalence between AVE~\eqref{eq:ave} and GLCP~\eqref{eq:ave2glcp}, and

  \item a property of the projection operator onto the nonnegative orthant.
\end{itemize}

However, the two properties are lacking for GAVE~\eqref{eq:gave} since GAVE~\eqref{eq:gave} cannot be reformulated as a GLCP or a nonlinear projection equation. Fortunately, by skipping the two properties, we can prove the following Theorem~\ref{thm:wucha1} and Theorem~\ref{thm:wucha2} for GAVE~\eqref{eq:gave}, which are counterparts of \cite[Theorem~3.5]{cyyh2021}  and~\cite[Theorem~4.1]{cyyh2021}, respectively. Then we can construct an inverse-free and globally fixed-time stable dynamical system for solving GAVE~\eqref{eq:gave}.

The rest of this paper is organized as follows. In Section \ref{sec2:pre} we state a few basic results on GAVE~\eqref{eq:gave} and the dynamic system, which are relevant to our later developments. A fixed-time dynamic system to solve GAVE~\eqref{eq:gave} is developed in Section \ref{sec3:main} and its convergence analysis is also given there. In Section~\ref{sec:fed}, the forward-Euler discretization of the proposed model is studied. Numerical simulations are given in Section \ref{sec4:numerical}.
Conclusions are made in Section \ref{sec5:conclusion}.

\textbf{Notation.} We use $\mathbb{R}^{n\times n}$ to denote the set of all $n \times n$ real matrices and $\mathbb{R}^{n}= \mathbb{R}^{n\times 1}$. We use $\mathbb{R}_+$ to denote the nonnegative reals. $I$ is the identity matrix with suitable dimension. $| \cdot |$ denotes absolute value for real scalar. The transposition of a matrix or a vector is denoted by $\cdot ^\top$. The inner product of two vectors in $\mathbb{R}^n$ is defined as $\langle x, y\rangle\doteq x^\top y= \sum\limits_{i=1}^n x_i y_i$ and $\| x \|\doteq\sqrt{\langle x, x\rangle} $ denotes the Euclidean norm of the vector $x\in \mathbb{R}^{n}$. $\|A\|$ denotes the spectral norm of $A$ and is defined by the formula $\| A \|\doteq \max \left\{ \| A x \| : x \in \mathbb{R}^{n}, \|x\|=1 \right\}$. The smallest singular value and the largest singular value  of~$A$ are denoted by $\sigma_{\min}(A)$ and $\sigma_{\max}(A)$, respectively. The projection mapping from $\mathbb{R}^n$ onto $\Omega$, denoted by $P_{\Omega}$, is defined as $P_{\Omega}[x]=\arg\min\{\|x-y\|:y\in \Omega\}$. A function $\rho: \mathbb{R}_+ \rightarrow  \mathbb{R}_+$ is said to belong to class $\mathcal{K}_\infty$ if it is continuous, zero at zero, strictly increasing, and unbounded. Given a set $S$, ${\rm conv} S$ denotes its convex hull, and $\mathbb{B}$ denotes the closed unit ball in a Euclidean space.

\section{Preliminaries}\label{sec2:pre}
For subsequent discussions, in this section we introduce some basic properties of GAVE~\eqref{eq:gave} and dynamical systems.

\begin{lemma}[{\rm\!\!{\cite[Theorem~2.1]{wuli2020}}}]\label{lem:gaveunique}
Suppose that $A,B\in\mathbb{R}^{n\times n}$ and~$\sigma_{\min}(A)>\|B\|$.  Then GAVE~\eqref{eq:gave} has a unique solution for any $c\in\mathbb{R}^n$.
\end{lemma}

\begin{theorem}\label{thm:wucha1}
Let~$A,B\in\mathbb{R}^{n\times n}$ and~$c\in\mathbb{R}^n$. If~$\sigma_{\min}(A)>\|B\|$, then for any $x\in\mathbb{R}^n$ we have
\begin{equation}\label{neq:wucha1}
(x-x_*)^\top A^\top (Ax - B|x| - c)\geq \frac{1}{2}\|Ax - B|x| - c\|^2,
\end{equation}
where~$x_*$ is the unique solution to~{\rm GAVE}~\eqref{eq:gave}.
\end{theorem}
\begin{proof}
Lemma~\ref{lem:gaveunique} and~$\sigma_{\min}(A)>\|B\|$ imply that GAVE~\eqref{eq:gave} has a unique solution $x_*$ for any $c\in\mathbb{R}^n$. Since $Ax_* - B|x_*| - c = 0$ and $\||x| - |x_*|\|\le \|x - x_*\|$, for any~$x\in\mathbb{R}^n$  we have
\begin{align*}
   (x- &x_*)^\top A^\top (Ax - B|x| - c) - \frac{1}{2}\|Ax - B|x| - c\|^2\\
    & = (x - x_*)^\top A^\top(Ax - B|x| - Ax_* + B|x_*|) - \frac{1}{2}\|Ax - B|x| - Ax_* + B|x_*|\|^2  \\
    & = (x - x_*)^\top A^\top A(x - x_*) - (x - x_*)^\top A^\top B(|x| - |x_*|) - \frac{1}{2}\|A(x - x_*)\|^2 \\
    & \qquad- \frac{1}{2}\|B(|x| - |x_*|)\|^2 + (x - x_*)^\top A^\top B(|x| - |x_*|)\\
    & = \frac{1}{2}\|A(x - x_*)\|^2 - \frac{1}{2}\|B(|x| - |x_*|)\|^2\\
    & \geq \frac{\sigma_{\min}^2(A)}{2}\|x - x_*\|^2 - \frac{\|B\|^2}{2}\|x - x_*\|^2\\
    & = \frac{\sigma_{\min}^2(A) - \|B\|^2}{2}\|x - x_*\|^2\\
    &\geq0,
\end{align*}
in which the last inequality follows from~$\sigma_{\min}(A) > \|B\|$ and~$\|x-x_*\|\geq0$.\hfill
\end{proof}
If~$B=I$, Theorem~\ref{thm:wucha1} reduces to the following Corollary~\ref{lem:wucha1}, which is a part of \cite[Theorem~3.5]{cyyh2021}. However, our proof here is differ from that of \cite[Theorem~3.5]{cyyh2021}. Specifically, the proof of \cite[Theorem~3.5]{cyyh2021} leverages the properties of GLCP~\eqref{eq:ave2glcp} and the projection mapping onto the nonnegative orthant, which are not used in the proof of  Theorem~\ref{thm:wucha1}.

\begin{corollary}\label{lem:wucha1}
Let~$A\in\mathbb{R}^{n\times n}$ and~$c\in\mathbb{R}^n$. If~$\sigma_{\min} (A) > 1$, then for any $x\in\mathbb{R}^n$ we have
\begin{equation*}\label{neq:wucha0}
(x - x_*)^\top A^\top (Ax - |x| - c)\geq\dfrac{1}{2}\|Ax - |x| - c\|^2,
\end{equation*}
where~$x_*$ is the unique solution to~AVE~\eqref{eq:ave}.
\end{corollary}

\begin{theorem}\label{thm:wucha2}
Let~$A,B\in\mathbb{R}^{n\times n}$ and~$c\in\mathbb{R}^n$. If~$\sigma_{\min}(A)>\|B\|$, then for any $x\in\mathbb{R}^n$ we have
\begin{equation}\label{neq:wucha2}
\frac{1}{\|A\|+\|B\|}\|Ax - B|x| - c\|\leq\|x - x_*\|\leq\frac{1}{\sigma_{\min}(A) - \|B\|}\|Ax - B|x| - c\|,
\end{equation}
where~$x_*$ is the unique solution to~GAVE~\eqref{eq:gave}.
\end{theorem}
\begin{proof}
It follows from Lemma~\ref{lem:gaveunique} and~$\sigma_{\min}(A)>\|B\|$ that GAVE~\eqref{eq:gave} has a unique solution $x_*$ for any $c\in\mathbb{R}^n$. Since $Ax_* - B|x_*| = c$ and $\||x| - |x_*|\|\le \|x - x_*\|$, for any~$x\in\mathbb{R}^n$, we have
\begin{equation}\label{nr:upbound}
\|Ax - B|x| - c\| = \|A(x - x_*) - B(|x| - |x_*|)\|\leq (\|A\| + \|B\|)\|x - x_*\|
\end{equation}
and
\begin{align}\nonumber
\|Ax - B|x| - c\| &= \|A(x - x_*) - B(|x| - |x_*|)\|\\\nonumber
&\geq \|A(x - x_*)\| - \|B(|x| - |x_*|)\|\\\label{nr:dombound}
&\geq(\sigma_{\min}(A) - \|B\|)\|x - x_*\|.
\end{align}
Then \eqref{neq:wucha2} follows from $\sigma_{\min}(A)>\|B\|$, $\|A\|+\|B\|>0$, \eqref{nr:upbound} and~\eqref{nr:dombound}.
\end{proof}

\begin{remark}\label{rem:wucha2}{\rm
If $B=I$, then Theorem~\ref{thm:wucha2} reduces  to~\cite[Corollary~2.2]{zlzh2024}. As shown in \cite{zlzh2024}, in this case, the error bound~\eqref{neq:wucha2} is tighter than the one proposed by~\cite[Theorem~4.1]{cyyh2021}.}
\end{remark}

Let~$f:\mathbb{R}^n\rightarrow \mathbb{R}^n $~be a continuous vector-valued function.  Consider the autonomous differential equation
\begin{equation}\label{eq:dynamic}
\frac{{\rm d}x}{{\rm d}t}=f(x).
\end{equation}
The solution of \eqref{eq:dynamic} with $x(0) = x_0$ is denoted by $x(t;x_0)$.

\begin{definition}{\rm(\!\!{\cite[p. 3]{khalil1996}})}
A point $x_*\in\mathbb{R}^n$ is said to be an equilibrium point of \eqref{eq:dynamic} if $f(x_*)=0$.
\end{definition}

\begin{lemma}[{\rm\!\!{\cite[Lemma~1]{poly2012}}}]\label{lem:ft}
Let~$x_*\in \mathbb{R}^n$ be an equilibrium point of \eqref{eq:dynamic}. If there exists a radially unbounded continuous function~$V:\mathbb{R}^n\rightarrow \mathbb{R}_+$ such that
\begin{itemize}
  \item [{\rm(1)}] $V(x) = 0 ~\Rightarrow x = x_*;$

  \item [{\rm(2)}] any solution~$x(t;x_{0})$ of \eqref{eq:dynamic} satisfies
$$
\frac{{\rm d}V(x(t;x_{0}))}{{\rm d}t} \le -\alpha V(x(t;x_{0}))^{\kappa_1} - \beta V(x(t;x_{0}))^{\kappa_2}
$$ for some  $\alpha>0,\beta>0$, $0<\kappa_1<1$, and $\kappa_2>1$.
\end{itemize}
Then the equilibrium point $x_*$ of  \eqref{eq:dynamic} is globally fixed-time stable with
$$
T(x_{0}) \le T_{\max} = \frac{1}{\alpha(1-\kappa_1)} + \frac{1}{\beta(\kappa_2 - 1)},~\forall x_{0}\in \mathbb{R}^n.
$$
\end{lemma}

\begin{lemma}[{\rm\!\!{\cite[Proposition~1]{gbgb2022}}}]\label{lem:unique}
Let $f:\mathbb{R}^n\rightarrow \mathbb{R}^n$ be a locally Lipschitz continuous vector-valued function such that
\begin{equation*}
f(x_*)=0\quad\text{and}\quad \langle x - x_*,f(x)\rangle>0,\forall x\in\mathbb{R}^n\setminus\{x_*\},
\end{equation*}
where~$x_*\in\mathbb{R}^n$. Consider the following autonomous differential equation
\begin{equation}\label{system:solu}
\frac{{\rm d}x}{{\rm d}t}=-\rho(x)f(x),
\end{equation}
where
$$\rho(x):=
\left\{
  \begin{array}{ll}
  \frac{\rho_1}{\|f(x)\|^{1-\lambda_1}} + \frac{\rho_2}{\|f(x)\|^{1-\lambda_2}}, & \text{if}~f(x)\neq0,\\
    0, & \text{if}~f(x)=0
  \end{array}
\right.
$$
with~$\rho_1,~\rho_2>0$, $\lambda_1\in(0,1)$ and $\lambda_2>1$. Then, the right-hand side of~\eqref{system:solu} is continuous for all~$x\in\mathbb{R}^n$, and starting
from any given initial condition, a solution of~\eqref{system:solu} exists and is uniquely determined for all~$t\geq 0$.
\end{lemma}

In order to study the $(T,\epsilon)$-close discrete-time approximation of the continuous-time dynamical systems, we adopt the following definition \cite[Definition~3]{gbgb2022}, which are adapted from \cite[Definition~3.2]{sate2010}.

\begin{definition}[\!\!{\cite[Definition~3]{gbgb2022}}]\label{defn:close}
Let $T > 0$, $\epsilon > 0$ and $\eta > 0$ be given. Assume that $x_c: [0,T]\rightarrow \mathbb{R}^n$ is a solution  of the following autonomous differential equation
\begin{equation}\label{eq:ncs}
\frac{{\rm d} x}{dt} = X_c(x),\quad  x(0) = x_{c}(0),
\end{equation}
where $X_c: \mathbb{R}^n \rightarrow \mathbb{R}^n$ is a continuous vector-valued function. For the following autonomous difference equation
\begin{equation}\label{eq:nds}
x^{(k+1)} = X_d(x^{(k)}), \quad x^{(0)} = x_{d}(0),
\end{equation}
a solution of it is denoted by $x_d: \left\{0, 1, 2, \ldots, \lfloor \frac{T}{\eta}\rfloor\right\}\rightarrow \mathbb{R}^n$. Here, $X_d: \mathbb{R}^n \rightarrow \mathbb{R}^n$ is also a continuous vector-valued function.
The solutions $x_c$ and $x_d$  are said to be $(T, \epsilon)$-close if
\begin{itemize}
  \item[{\rm (i)}] for each $t\in [0, T]$, there exists a $k\in \left\{0,1, 2, \ldots, \lfloor \frac{T}{\eta}\rfloor \right\}$ such that $|t - \eta k| < \epsilon$ and $\|x_c(t) - x_d(k)\| < \epsilon${\rm;}

  \item[{\rm (ii)}] for each $k\in\left\{0,1, 2, \ldots, \lfloor \frac{T}{\eta}\rfloor \right\}$, there exists a $t\in [0, T]$ such that $|t - k\eta| < \epsilon$ and $\|x_c(t) - x_d {(k)}\| < \epsilon$.
\end{itemize}
\end{definition}

Inspired by \cite[Lemma~3.1]{zhbr2025}, for closeness of continuous and discrete solutions, we recall \cite[Theorem~5.2]{sate2010}, which adapting it to our special case of a hybrid system with one differential equation.

\begin{lemma}[{\cite[Theorem~5.2]{sate2010}}]\label{lem:epsilon}
Consider the dynamic system \eqref{eq:ncs} and $X_c$ is assumed to be continuous and thus locally bounded in $\mathbb{R}^n$. Consider $X_d$ in \eqref{eq:nds} such that, for each compact set $K\subset \mathbb{R}^n$, there exists $\rho \in \mathcal{K}_\infty$ and $\eta_* > 0$ such that for each $x\in K$ and each $\eta \in (0,\eta_*]$,
\begin{equation}\label{eq:tec}
X_{d}(x)\in x + \eta {\rm conv} X_c(x + \rho(\eta)\mathbb{B}) + \eta\rho(\eta)\mathbb{B}.
\end{equation}
Then, for a compact set $K\subset\mathbb{R}^n$, $\forall\epsilon>0$, and $ \forall T\geq0$, there exists $\eta_* >0$ such that: $\forall\eta\in(0, \eta_*]$ and any discrete solution $x_d$ with $x^{(0)}\in K + \delta\mathbb{B}$ for some $\delta > 0$, there exists a continuous solution $x_c(t)$ with $x(0)\in K$ such that $x_d$ and $x_c$ are $(T, \epsilon)$-close.
\end{lemma}

\section{Fixed-time stable dynamical system for solving GAVE~\eqref{eq:gave}}\label{sec3:main}
In this section, we establish the following fixed-time stable and inverse-free dynamic model for solving GAVE~\eqref{eq:gave}:
\begin{equation}\label{stem:model53}
\frac{{\rm d}x}{{\rm d}t} = -\rho(x)g(\gamma, x) ,
\end{equation}
where
\begin{equation}\label{rho}
\rho(x)=\left\{
                      \begin{array}{ll}
                       \frac{\rho_1}{\|g(\gamma, x)\|^{1-\lambda_1}} + \frac{\rho_2}{\|g(\gamma, x)\|^{1-\lambda_2}}, &  \text{if}~g(\gamma, x)\neq 0, \\
                        0, &  \text{if}~g(\gamma, x) = 0,
                      \end{array}
                    \right.
\end{equation}
$g(\gamma,x) = \gamma A^\top r(x)$, $r(x)=Ax - B|x| - c$,  $\rho_1, \rho_2, \gamma>0$, $\lambda_1\in(0,1)$, and $\lambda_2\in(1,+\infty)$.

\begin{theorem}\label{equi}
Suppose that $A,B\in\mathbb{R}^{n\times n}$ and  $\sigma_{\min}(A)>\|B\|$. Then the dynamic model~\eqref{eq:dynamic} has a unique equilibrium point~$x_*$ for any $c\in \mathbb{R}^n$, which is the unique solution of  GAVE~\eqref{eq:gave}.
\end{theorem}
\begin{proof}
If~$x_*$ is an equilibrium point of \eqref{stem:model53}, that is
\begin{equation*}
\rho(x_*)A^\top r(x_*)=0.
\end{equation*}
Since~$\sigma_{\min}(A)>\|B\|$, then~$A$ is invertible and the above equation implies
\begin{equation*}
\rho(x_*)=0\quad\text{or}\quad r(x_*)=0,
\end{equation*}
which together with \eqref{rho} implies
\begin{equation*}
r(x_*)=0,
\end{equation*}
i.e., $x_*$ is a solution to GAVE~\eqref{eq:gave}. If $x_*$ is a solution of GAVE~\eqref{eq:gave}, then it is also an equilibrium point of \eqref{stem:model53}. Hence, $x_*$ is an equilibrium point of \eqref{stem:model53} if and only if it is a solution to GAVE~\eqref{eq:gave}.

Lemma~\ref{lem:gaveunique} implies that GAVE~\eqref{eq:gave} has a unique solution $x_*$ for any $c\in \mathbb{R}^n$ whenever  $\sigma_{\min}(A)>\|B\|$.
\end{proof}

\begin{theorem}\label{thm:L}
For any given $\gamma > 0$, the function $g(\gamma,x)$ defined by \eqref{stem:model53} is Lipschitz continuous on~$\mathbb{R}^n$.
\end{theorem}
\begin{proof}
From~\eqref{stem:model53} and the inequality~$\||x| - |y|\|\leq\|x - y\|$, for any~$x,y\in\mathbb{R}^n$, we have
\begin{align*}
\|g(\gamma,x) - g(\gamma,y)\| &= \|\gamma A^\top (Ax - B|x| - c) - \gamma A^\top (Ay - B|y| - c)\|\\
                              &= \gamma\|A^\top A(x - y) - A^\top B(|x| - |y|)\|\\
                              &\leq\gamma\left(\|A^\top A\|+\|A^\top B\|\right)\|x - y\|.
\end{align*}
Hence, for any given $\gamma > 0$, $g(\gamma,x)$ is Lipschitz continuous on~$\mathbb{R}^n$ with Lipschitz constant~$\gamma\left(\|A^\top A\|+\|A^\top B\|\right)$.
\end{proof}

Combine Theorem~\ref{thm:wucha1}, Lemma~\ref{lem:unique}, Theorem~\ref{equi} and Theorem~\ref{thm:L}, we obtain the following theorem.

\begin{theorem}\label{thm:unique}
Let~$A, B\in\mathbb{R}^{n\times n}$ and $c\in \mathbb{R}^n$. If~$\sigma_{\min}(A)>\|B\|$, then for any given initial condition~$x(0) = x_{0}$, the dynamic model~\eqref{stem:model53} has a unique solution~$x(t;x_{0})$ with~$t\in[0,+\infty)$.
\end{theorem}

Now we are in the position to explore the stability for the equilibrium point of the proposed model~\eqref{stem:model53}.

\begin{theorem}\label{thm:gconv}
Let $A, B\in\mathbb{R}^{n\times n}$, $c\in \mathbb{R}^n$ and $\sigma_{\min}(A)>\|B\|$. Then the unique equilibrium point of \eqref{stem:model53} is globally fixed-time stable with the settling-time satisfying
\begin{equation}\label{t:tmax}
T(x_{0})\leq T_{\max}=\dfrac{1}{c_1(1-\kappa_1)}+\dfrac{1}{c_2(\kappa_2-1)},\quad \forall x(0)=x_{0}\in \mathbb{R}^n,
\end{equation}
where $c_1$, $c_2$, $\kappa_1$ and~$\kappa_2$ are defined as in \eqref{bound}. Here, $T:\mathbb{R}^n \rightarrow \mathbb{R}_+$ is the settling-time function.
\end{theorem}
\begin{proof}
Since $\sigma_{\min}(A)>\|B\|$, it follows from Theorem~\ref{equi}  that  \eqref{stem:model53} has a unique equilibrium point $x_*$.

Define
\begin{equation}\label{f:ly}
V(x)=\dfrac{1}{2}\|x - x_*\|^2.
\end{equation}
By~\eqref{f:ly}, we conclude that~$V(x)\rightarrow\infty$ as~$\|x - x_*\|\rightarrow\infty$ and~$V(x)=0$ if and only if~$x=x_*$. Given any~$x(0)=x_{0}\in\mathbb{R}^n\setminus\{x_*\}$, Theorem~\ref{thm:unique} implies that the proposed model~\eqref{stem:model53} has a unique solution~$x=x(t;x_0)$ with $t\ge 0$. Then it follows from~\eqref{rho} and~\eqref{f:ly} that
\begin{align}\nonumber
  \frac{{\rm d}V(x)}{{\rm d}t} &=(x - x_*)^\top\frac{{\rm d}x}{{\rm d}t}=-\left\langle x - x_*,\gamma\rho(x)A^\top r(x)\right\rangle \\\nonumber
   & =-\left\langle x - x_*,\frac{\gamma\rho_1A^\top r(x)}{\|g(\gamma,x)\|^{1 - \lambda_1}} + \frac{\gamma\rho_2A^\top r(x)}{\|g(\gamma,x)\|^{1 - \lambda_2}} \right\rangle\\\label{eq:dv}
   & =-\frac{\gamma\rho_1}{\|g(\gamma,x)\|^{1 - \lambda_1}}\langle x - x_*,A^\top r(x)\rangle -\frac{\gamma\rho_2}{\|g(\gamma,x)\|^{1 - \lambda_2}}\langle x - x_*,A^\top r(x)\rangle.
\end{align}
Apply \eqref{neq:wucha1}, the second inequality of~\eqref{neq:wucha2} and~$\sigma_{\min}(A)>\|B\|$, we obtain
\begin{align}\nonumber
\frac{\gamma\rho_1}{\|g(\gamma,x)\|^{1 - \lambda_1}}\langle x - x_*,A^\top r(x)\rangle&\geq \frac{\gamma\rho_1}{2\|g(\gamma,x)\|^{1 - \lambda_1}}\|r(x)\|^2\geq\frac{\gamma^{\lambda_1}\rho_1\|A\|^{\lambda_1-1}}{2}\|r(x)\|^{\lambda_1+1}\\\label{eq:you1}
&\geq\frac{\gamma^{\lambda_1}\rho_1\|A\|^{\lambda_1-1}(\sigma_{\min}(A) - \|B\|)^{\lambda_1 + 1}}{2}\|x - x_*\|^{\lambda_1 + 1}
\end{align}
and
\begin{align}\nonumber
\frac{\gamma\rho_2}{\|g(\gamma,x)\|^{1 - \lambda_2}}\langle x - x_*,A^\top r(x)\rangle&\geq \frac{\gamma\rho_2}{2\|g(\gamma,x)\|^{1 - \lambda_2}}\|r(x)\|^2\geq\frac{\gamma^{\lambda_2}\rho_2\sigma_{\min}(A)^{\lambda_2-1}}{2}\|r(x)\|^{\lambda_2+1}\\\label{eq:you}
&\geq\frac{\gamma^{\lambda_2}\rho_2\sigma_{\min}(A)^{\lambda_2-1}(\sigma_{\min}(A) - \|B\|)^{\lambda_2 + 1}}{2}\|x - x_*\|^{\lambda_2 + 1}.
\end{align}
From~\eqref{eq:dv}--\eqref{eq:you}, we have
\begin{align}\small\nonumber
  \frac{{\rm d}V(x)}{{\rm d}t} &\leq-\frac{\gamma^{\lambda_1}\rho_1\|A\|^{\lambda_1-1}(\sigma_{\min}(A) - \|B\|)^{\lambda_1 + 1}}{2}\|x - x_*\|^{\lambda_1 + 1}\\\nonumber
    &\quad -\frac{\gamma^{\lambda_2}\rho_2\sigma_{\min}(A)^{\lambda_2-1}(\sigma_{\min}(A) - \|B\|)^{\lambda_2 + 1}}{2}\|x - x_*\|^{\lambda_2 + 1}\\\nonumber
    &=-2^{\frac{\lambda_1-1}{2}}\gamma^{\lambda_1}\rho_1\|A\|^{\lambda_1-1}(\sigma_{\min}(A) - \|B\|)^{\lambda_1 + 1}\left(\frac{1}{2}\|x - x_*\|^2\right)^{\frac{\lambda_1+1}{2}} \\\nonumber
   &\quad -2^{\frac{\lambda_2-1}{2}}\gamma^{\lambda_2}\rho_2\sigma_{\min}(A)^{\lambda_2-1}(\sigma_{\min}(A) - \|B\|)^{\lambda_2 + 1}\left(\frac{1}{2}\|x - x_*\|^2\right)^{\frac{\lambda_2+1}{2}} \\\label{neq:v(x)}
   &= -c_1V(x)^{\kappa_1} - c_2V(x)^{\kappa_2},
\end{align}
where
\begin{equation}\label{bound}
\begin{array}{cc}
   c_1 = 2^{\frac{\lambda_1-1}{2}}\gamma^{\lambda_1}\rho_1\|A\|^{\lambda_1-1}(\sigma_{\min}(A) - \|B\|)^{\lambda_1 + 1}>0,& \kappa_1 = \frac{\lambda_1+1}{2}\in(0.5,1), \\
   \quad~~ c_2 =2^{\frac{\lambda_2-1}{2}}\gamma^{\lambda_2}\rho_2\sigma_{\min}(A)^{\lambda_2-1}(\sigma_{\min}(A) - \|B\|)^{\lambda_2 + 1}>0, & ~\kappa_2 = \frac{\lambda_2+1}{2}\in(1,+\infty).
\end{array}
\end{equation}
Then the proof is completed by Lemma~\ref{lem:ft}.
\end{proof}

Theorem~\ref{thm:gconv} establishes the fixed-time stability of the equilibrium point for the dynamical model~\eqref{stem:model53}. From~\eqref{t:tmax} and \eqref{bound}, the upper bound $T_{\max}$ of the settling time for the equilibrium point depends on the parameters $\gamma$, $\rho_1$, $\rho_2$, $\lambda_1$ and $\lambda_2$. As shown in the following corollary, by specially choosing $\rho_1$, $\rho_2$, $\lambda_1$, and $\lambda_2$, the upper bound $T_{\max}$ of the settling time for the equilibrium point can be directly determined by the value of $\gamma$.

\begin{corollary}\label{cor:xi}
Let $c\in \mathbb{R}^n$ and  $A,B\in\mathbb{R}^{n\times n}$ with $\sigma_{\min}(A)>\|B\|$. Given any $\xi >1$, let $\lambda_1=1 - \frac{1}{\xi}$, $\lambda_2=1 + \frac{1}{\xi}$, and $\rho_1\rho_2=\frac{\xi^2\pi^2}{\left(\frac{\sigma_{\min}(A)}{\|A\|}\right)^{\frac{1}{\xi}}(\sigma_{\min}(A) - \|B\|)^4}$.  Then, the unique equilibrium point $x_*$ of the system \eqref{stem:model53} is globally fixed-time stable with the settling-time satisfying
\begin{equation*}\label{ie:tg}
T(x_{0})\leq T_{\max}< \dfrac{1}{\gamma}.
\end{equation*}
\end{corollary}
\begin{proof}
Under the conditions of the corollary, the inequality \eqref{neq:v(x)} holds. Given any $\xi >1$,  substituting $\lambda_1=1 - \frac{1}{\xi}$ and $\lambda_2=1 + \frac{1}{\xi}$ into \eqref{neq:v(x)}, we have
\begin{equation} \label{eq:vx}
  \frac{{\rm d}V(x)}{{\rm d}t} \leq -c_1V(x)^{\frac{2\xi-1}{2\xi}} - c_2V(x)^{\frac{2\xi+1}{2\xi}}.
\end{equation}

Let $z=V(x)^{-\frac{1}{2\xi}}>0$. Then $V(x)=z^{-2\xi}$, and
\begin{equation}\label{neq:weifen}
\frac{{\rm d}V(x)}{{\rm d}t}=-2\xi z^{-2\xi-1}\frac{{\rm d}z}{{\rm d}t}=-2\xi V(x)^{\frac{2\xi+1}{2\xi}}\frac{{\rm d}z}{{\rm d}t}.
\end{equation}
Substituting~\eqref{neq:weifen} into~\eqref{eq:vx}, we have
\begin{equation*}
-2\xi V(x)^{\frac{2\xi+1}{2\xi}}\frac{{\rm d}z}{{\rm d}t}\leq-c_1V(x)^{\frac{2\xi-1}{2\xi}} - c_2V(x)^{\frac{2\xi+1}{2\xi}},
\end{equation*}
that is,
\begin{equation*}\label{eq:itegr}
2\xi \frac{{\rm d}z}{{\rm d}t}\geq c_1V(x)^{-\frac{1}{\xi}} + c_2=c_1 z^2+c_2.
\end{equation*}

For any given $T_{\max}\geq0$, we have
\begin{equation*}\label{weifen1}
\frac{2\xi}{c_2}\int_{z(0)}^{+\infty} \frac{{\rm d}z}{1 + \left(\sqrt{\frac{c_1}{c_2}}z\right)^2}\ge\int_0^{T_{\max}} {\rm d}t,
\end{equation*}
which implies that
\begin{equation*}\label{con:c}
T_{\max}\leq\frac{2\xi\left(\frac{\pi}{2} - \arctan(\sqrt{\frac{c_1}{c_2}}z(0))\right)}{\sqrt{c_1c_2}}< \frac{\xi\pi}{\sqrt{c_1c_2}},
\end{equation*}
in which the second inequality follows from $z(0) = V(x(0))^{-\frac{1}{2\xi}} >0$.

Since~$\lambda_1=1 - \frac{1}{\xi}$ and~$\lambda_2=1 + \frac{1}{\xi}$, then we have $c_1c_2=\gamma^2\rho_1\rho_2\|A\|^{ - \frac{1}{\xi}}\sigma_{\min}(A)^{\frac{1}{\xi}}(\sigma_{\min}(A) - \|B\|)^4$. Hence, if $\rho_1\rho_2=\frac{\xi^2\pi^2}{\left(\frac{\sigma_{\min}(A)}{\|A\|}\right)^{\frac{1}{\xi}}(\sigma_{\min}(A) - \|B\|)^4}$, we have
$$
T_{\max}< \dfrac{\xi\pi}{\sqrt{c_1c_2}}=\dfrac{\xi\pi}{\gamma\left(\dfrac{\sigma_{\min}(A)}{\|A\|}\right)^{\frac{1}{2\xi}}\sqrt{\rho_1\rho_2}(\sigma_{\min}(A) - \|B\|)^2}=\dfrac{1}{\gamma}.
 $$
This proof is completed.
\end{proof}

\section{$(T,\epsilon)$-close discrete-time approximation scheme}\label{sec:fed}
Continuous-time dynamical systems provide a natural and intuitive way to speed up algorithms. However, in practice, a discrete-time implementation is used \cite{gbgb2022}. In general, the fixed-time convergence behavior of the continuous-time dynamical system might not be preserved in the discrete-time version. A consistent discrete-time approximation scheme preserves the convergence behavior of the continuous-time dynamical system in the discrete-time setting \cite{gbgb2022}.  Polyakov et al. \cite{poeb2019} present a consistent semi-implicit discretization for practically fixed-time stable systems. Zhang et al. \cite{zhbr2025} show the closeness between solutions for the proposed continuous flows and the trajectories of their forward Euler discretization. Inspired by \cite{zhbr2025,gbgb2022}, this section gives sufficient conditions that lead to an explicit $(T,\epsilon)$-close discrete-time approximation scheme for the fixed-time stable system \eqref{stem:model53}.

By using the forward-Euler discretization of \eqref{stem:model53}, we propose the following iteration method for solving GAVE~\eqref{eq:gave}:
\begin{equation}\label{eq:disc}
x^{(k+1)} = x^{(k)} - \eta\rho(x^{(k)})g(\gamma,x^{(k)}),\quad k = 0,1,2,\ldots,
\end{equation}
where~$\eta>0$ is the time-step, $\rho(x)$ and $g(\gamma,x)$ are defined as in \eqref{stem:model53}. For \eqref{eq:disc}, define $x_d: \{0,1,2,\ldots,\}\rightarrow \mathbb{R}^n$ as
$$
x_d(i) = x^{(i)}, \quad i = 0,1,2,\ldots.
$$

\begin{theorem}\label{thm:disc}
Let $c\in \mathbb{R}^n$ and $A,B\in\mathbb{R}^{n\times n}$ with $\sigma_{\min}(A)>\|B\|$. For any given $\xi>1$, assume that $\lambda_1=1 - \frac{1}{\xi}$, $\lambda_2=1 + \frac{1}{\xi}$ and  $\rho_1\rho_2=\frac{\xi^2\pi^2}{\left(\frac{\sigma_{\min}(A)}{\|A\|}\right)^{\frac{1}{\xi}}(\sigma_{\min}(A) - \|B\|)^4}$. Then, for any given $x^{(0)} \in \mathbb{R}^n$ and $\epsilon>0$, there exists $\eta_* > 0$ such that for any $\eta \in (0,\eta_*]$, we have
\begin{equation}\label{eq:xbar}
\|x^{(k)} -x_*\|\leq
\begin{cases}\small
                   \begin{array}{ll}
                     \sqrt{2}\left(\frac{\xi\pi}{\rho_2(\sigma_{\min}(A)-\|B\|)^2}\left(\frac{1}{\sqrt{2}\gamma\sigma_{\min}(A)(\sigma_{\min}(A)-\|B\|)}\right)^{\frac{1}{\xi}}
                     \tan\left(\frac{\pi}{2} -\frac{\pi\eta\gamma}{2} k\right)\right)^\xi+\epsilon, &k\le k_*, \\
                     \epsilon, & \text{otherwise},
                   \end{array}
\end{cases}
\end{equation}
where $k_* = \left \lceil \frac{1}{\eta\gamma}\right\rceil$ and $x_*$ is the unique equilibrium point of~\eqref{stem:model53}.
\end{theorem}

\begin{proof}
The proof is inspired by those of \cite[Theorem 2]{gbgb2022} and \cite[Theorem~3.2]{zhbr2025}.

Select $x^{(0)}=x(0)\in \mathbb{R}^n$ and $x(0)\neq x_*$, there is a unique solution $x = x(t; x(0))$ of \eqref{stem:model53} with $t\ge 0$. Once again, let $z=V(x)^{-\frac{1}{2\xi}}$ with $V(x)$ being defined by \eqref{f:ly}.

It follows from the proof of Corollary~\ref{cor:xi} that
\begin{equation}\label{weifen1}
\frac{2\xi}{c_2}\frac{{\rm d}z}{1 + \left(\sqrt{\frac{c_1}{c_2}}z\right)^2}\geq{\rm d}t.
\end{equation}

For any given $T>0$, integrating both sides of \eqref{weifen1} yields
\begin{equation}\label{weifen2}
\frac{2\xi}{c_2}\int_{z(0)}^{z(T)} \frac{{\rm d}z}{1 + \left(\sqrt{\frac{c_1}{c_2}}z\right)^2}\geq\int_0^T {\rm d}t.
\end{equation}
From \eqref{weifen2} and $z=V(x)^{-\frac{1}{2\xi}}$, we have
\begin{align}\nonumber
V(x(T))&\leq\left(\sqrt{\frac{c_1}{c_2}}\frac{1}{\tan\left(\frac{\sqrt{c_1c_2}}{2\xi}(T+C)\right)}\right)^{2\xi}\\\nonumber
&=\left(\sqrt{\frac{c_1}{c_2}}\cdot\frac{1}{\frac{\tan\big(\frac{\sqrt{c_1c_2}}{2\xi}T\big)+\tan\big(\frac{\sqrt{c_1c_2}}{2\xi}C\big)}{1 - \tan\big(\frac{\sqrt{c_1c_2}}{2\xi}T\big)\tan\big(\frac{\sqrt{c_1c_2}}{2\xi}C\big)}}\right)^{2\xi}\\\label{neq:vx}
      &=\left(\sqrt{\frac{c_1}{c_2}}\cdot\frac{1 -\tan\big(\frac{\sqrt{c_1c_2}}{2\xi}T\big)\tan\big(\frac{\sqrt{c_1c_2}}{2\xi}C\big)}{\tan\big(\frac{\sqrt{c_1c_2}}{2\xi}T\big)
      +\tan\big(\frac{\sqrt{c_1c_2}}{2\xi}C\big)}\right)^{2\xi},
\end{align}
where
\begin{equation*}\label{eq:c}
C=\frac{2\xi}{\sqrt{c_1c_2}}\arctan\left(\sqrt{\frac{c_1}{c_2}}V(x(0))^{-\frac{1}{2\xi}}\right).
\end{equation*}
Since
\begin{equation*}
\tan\left(\frac{\sqrt{c_1c_2}}{2\xi}C\right)=\tan\left(\frac{\sqrt{c_1c_2}}{2\xi}\cdot\frac{2\xi}{\sqrt{c_1c_2}}\arctan\left(\sqrt{\frac{c_1}{c_2}}V(x(0))^{-\frac{1}{2\xi}}\right)\right)
=\sqrt{\frac{c_1}{c_2}}V(x(0))^{-\frac{1}{2\xi}}>0,
\end{equation*}
it follows from  \eqref{neq:vx} that
\begin{align}\nonumber
  V(x(T))&\leq\left(\sqrt{\frac{c_1}{c_2}}\cdot\frac{1 -\tan\big(\frac{\sqrt{c_1c_2}}{2\xi}T\big)\sqrt{\frac{c_1}{c_2}}V(x(0))^{-\frac{1}{2\xi}}}{\tan\big(\frac{\sqrt{c_1c_2}}{2\xi}T\big)
      +\sqrt{\frac{c_1}{c_2}}V(x(0))^{-\frac{1}{2\xi}}}\right)^{2\xi}\\\nonumber
      &=\left(\sqrt{\frac{c_1}{c_2}}\cdot\frac{\sqrt{\frac{c_2}{c_1}}V(x(0))^{\frac{1}{2\xi}} -\tan\big(\frac{\sqrt{c_1c_2}}{2\xi}T\big)}{1 + \sqrt{\frac{c_2}{c_1}}V(x(0))^{\frac{1}{2\xi}}\tan\big(\frac{\sqrt{c_1c_2}}{2\xi}T\big)}\right)^{2\xi}\\\nonumber
      &=\left(\sqrt{\frac{c_1}{c_2}}\tan\left(\arctan\left(\sqrt{\frac{c_2}{c_1}}V(x(0))^{\frac{1}{2\xi}}\right) -\frac{\sqrt{c_1c_2}}{2\xi}T\right)\right)^{2\xi}.\\\label{neq:x}
      &\le \left(\sqrt{\frac{c_1}{c_2}}\tan\left(\frac{\pi}{2} -\frac{\sqrt{c_1c_2}}{2\xi}T\right)\right)^{2\xi}.
\end{align}
If $T_{\max} = \frac{\pi \xi}{\sqrt{c_1c_2}}$, if follows from \eqref{neq:x} that $V(x(T_{\max})) = 0$. Moreover, it follows from \eqref{neq:v(x)} that $V(x(t)) \equiv 0, \forall t\ge T_{\max}$. Then, for any given $x(0)\in \mathbb{R}^n$, we have
\begin{equation}\label{eq:xtbar1}
\|x(t;x(0)) - x_*\|\leq\left\{
  \begin{array}{ll}
    \sqrt{2}\left(\sqrt{\frac{c_1}{c_2}}\tan\left(\frac{\pi}{2} -\frac{\sqrt{c_1c_2}}{2\xi}t\right)\right)^\xi, & 0\leq t<T_{\max}, \\
    0, & \text{otherwise}.
  \end{array}
\right.
\end{equation}
Since $\sqrt{c_1c_2}=\gamma\xi\pi$ and $\sqrt{\frac{c_1}{c_2}}=\frac{\xi\pi}{\rho_2(\sigma_{\min}(A)-\|B\|)^2}\left(\frac{1}{\sqrt{2}\gamma\sigma_{\min}(A)(\sigma_{\min}(A)-\|B\|)}\right)^{\frac{1}{\xi}}$, then from \eqref{eq:xtbar1} we have
\begin{equation}\label{eq:xtbar2}\small
\|x(t;x(0)) - x_*\|\leq\left\{
  \begin{array}{ll}
    \sqrt{2}\left(\frac{\xi\pi}{\rho_2(\sigma_{\min}(A)-\|B\|)^2}\left(\frac{1}{\sqrt{2}\gamma\sigma_{\min}(A)(\sigma_{\min}(A)-\|B\|)}\right)^{\frac{1}{\xi}}
    \tan\left(\frac{\pi}{2} -\frac{\gamma\pi}{2}t\right)\right)^\xi, & 0\leq t<\frac{1}{\gamma}, \\
    0, & \text{otherwise}.
  \end{array}
\right.
\end{equation}

From~\eqref{stem:model53} and~\eqref{rho}, we know that, as a function of $x$, $X_c(x) = -\rho(x)g(\gamma,x)$ is continuous and thus locally bounded in~$\mathbb{R}^n$. Now consider the forward-Euler discretization system~\eqref{eq:disc}. The mapping $X_d$ in this case is given by
$$
X_d(x) = x - \eta \rho(x)g(\gamma,x)
$$
with $\eta>0$, which clearly satisfies condition \eqref{eq:tec}. Hence, with Lemma~\ref{lem:epsilon} we can conclude about the $(T,\epsilon)$-closeness between the continuous-time solution $x(t;x(0))$ of \eqref{stem:model53} and the disctete-time solution $x_d(k)$ of \eqref{eq:disc}, provided that  $x^{(0)} = x(0)$ .

For any given $T\ge 0$, let $\tilde{k}_* = \lfloor \frac{T}{\eta}\rfloor$, according to Definition~\ref{defn:close}, for each $k\in \{0,1,2,\ldots, \tilde{k}_*\}$, there exits a $t\in [0,T]$ such that $|t-\eta k| < \epsilon$ and $\|x(t;x(0)) - x_d(k)\|<\epsilon$. Then for any $k\in \{0,1,2,\ldots\}$, we can conclude \eqref{eq:xbar} by  \eqref{eq:xtbar2} and $\|x^{(k)} - x_*\| \le \|x(t;x(0)) - x_*\|+ \|x^{(k)}-x(t;x(0))\|$ with $t = \eta k$.
\end{proof}

Next, we use a one-dimensional numerical example to demonstrate the closeness between the analytical solution of model \eqref{stem:model53} and its forward Euler discretization \eqref{eq:disc}.

\begin{example}\label{remark}
Consider GAVE \eqref{eq:gave} with $A=2$, $B=1$ and $c=1$, where $x\in\mathbb{R}$. It is obvious that $\sigma_{\min}(A)=2>\|B\|=1$, thus this GAVE has a unique solution $x_*=1$. For model \eqref{stem:model53} and method \eqref{eq:disc}, let $\gamma=10$, $\rho_1=1$, $\xi=2$ and $\rho_2=\frac{\xi^2\pi^2}{\left(\frac{\sigma_{\min}(A)}{\|A\|}\right)^{\frac{1}{\xi}}(\sigma_{\min}(A) - \|B\|)^4}=4\pi^2$. Then model \eqref{stem:model53} becomes
\begin{equation}\label{modelexam}
\frac{{\rm d}x}{{\rm d}t} = -20\cdot\rho(x)(2x - |x| - 1) ,
\end{equation}
where
\begin{equation}\label{rhoexam}
\rho(x)=\left\{
                      \begin{array}{ll}
                       \frac{1}{\sqrt{20\cdot|2x - |x| - 1|}} + 4\pi^2\sqrt{20\cdot|2x - |x| - 1|}, &  \text{if}~x\neq 1, \\
                        0, &  \text{if}~x = 1.
                      \end{array}
                    \right.
\end{equation}
Then the analytical solutions of model \eqref{modelexam} is:
\begin{itemize}
  \item[{\rm(1)}] $$
           x(t)=\left\{
                \begin{array}{ll}
                  1 + \dfrac{1}{80\pi^2}\tan^2\left(-20\pi t + \arctan\left(\sqrt{80(x(0)-1)}\cdot\pi\right)\right), & t\le T_{x(0)},  \\
                  1, & T_{x(0)}<t\le T,
                \end{array}
              \right.
      $$
where $x(0)>1$ and $T_{x(0)}=\dfrac{1}{20\pi}\arctan\left(\sqrt{80(x(0) - 1)}\cdot\pi\right)$;
  \item[{\rm(2)}] $$
           x(t)=\left\{
                \begin{array}{ll}
                  1 - \dfrac{1}{80\pi^2}\tan^2\left(-20\pi t + \arctan\left(\sqrt{80(1 - x(0))}\cdot\pi\right)\right), & t\le T_{x(0)},  \\
                  1, & T_{x(0)}<t\le T,
                \end{array}
              \right.
      $$
where $x(0)\in[0,1)$ and $T_{x(0)}=\dfrac{1}{20\pi}\arctan\left(\sqrt{80(1 - x(0))}\cdot\pi\right)$;
  \item[{\rm(3)}] $$
           x(t)=\left\{
                \begin{array}{ll}
                  \dfrac{1}{3} - \dfrac{1}{240\pi^2}\tan^2\left(-60\pi t + \arctan\left(\sqrt{80 - 240x(0)}\cdot\pi\right)\right), & t\le T_{x(0)},  \\
                  1 - \dfrac{1}{80\pi^2}\tan^2\left(-20\pi t + \arctan\left(\sqrt{80}\cdot\pi\right) + 20\pi T_{x(0)}\right), & T_{x(0)}<t\le\hat{T}, \\
                  1, & \hat{T}<t\le T,
                \end{array}
              \right.
      $$
where $x(0)<0$, $T_{x(0)}=\frac{\arctan\left(\sqrt{80 - 240x(0)}\cdot\pi\right) - \arctan\left(\sqrt{80}\cdot\pi\right)}{60\pi}$ and $\hat{T}=\frac{1}{20\pi}\arctan\left(\sqrt{80}\cdot\pi\right) + T_{x(0)}$.
\end{itemize}

For the system \eqref{modelexam}, we perform forward-Euler discretization with $\eta = 10^{-9}$. In this case, the discrete iterative method derived from \eqref{modelexam} is as follows:
\begin{equation}\label{eq:discexam}
x^{(k+1)} = x^{(k)} - 2\times10^{-8}\cdot\rho(x^{(k)})(2x^{(k)} - |x^{(k)}| - 1)~\text{for}~k\in\{0,1,2,\cdots,k_*\},
\end{equation}
in which $k_*= \lceil\frac{1}{10\eta}\rceil$ and $\rho(x)$ is defined as in \eqref{rhoexam}. As Figure \ref{figcompre} shows, the sequence $\{x^{(k)}\}_{k=0}^{k_*}$ generated by the method \eqref{eq:discexam} and the trajectory of $x(t)$ generated by the system \eqref{modelexam} satisfy Definition \ref{defn:close} with $t=k\eta$, that is $x^{(k)}$ and $x(t)$ are $(T,\epsilon)$-close with $\epsilon=10^{-5}$.

\begin{figure}[!htbp]
{\centering
\begin{tabular}{ccc}
\hspace{-0.15 cm}
\resizebox*{0.32\textwidth}{0.18\textheight}{\includegraphics{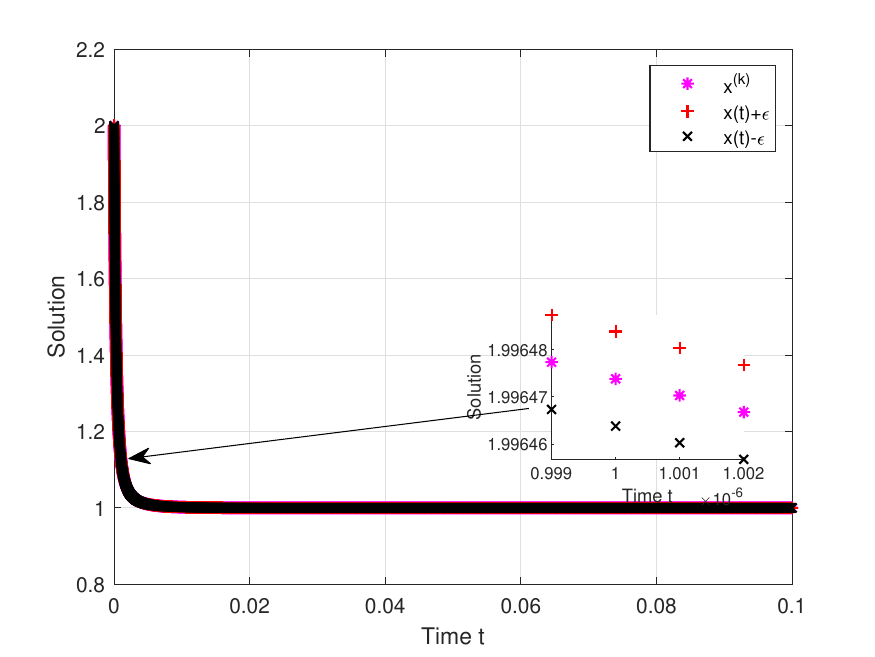}}
&\resizebox*{0.32\textwidth}{0.18\textheight}{\includegraphics{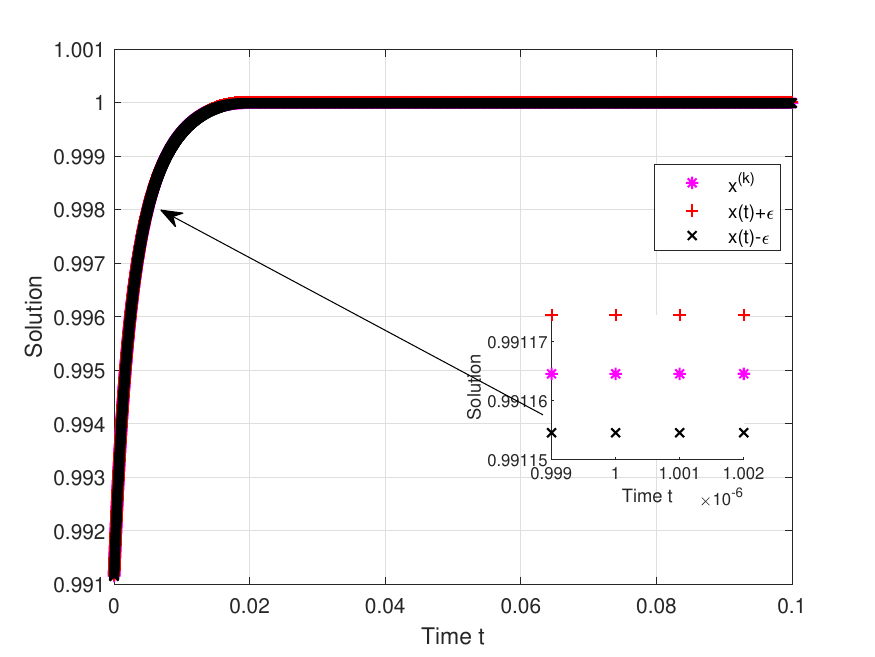}}
&\resizebox*{0.32\textwidth}{0.18\textheight}{\includegraphics{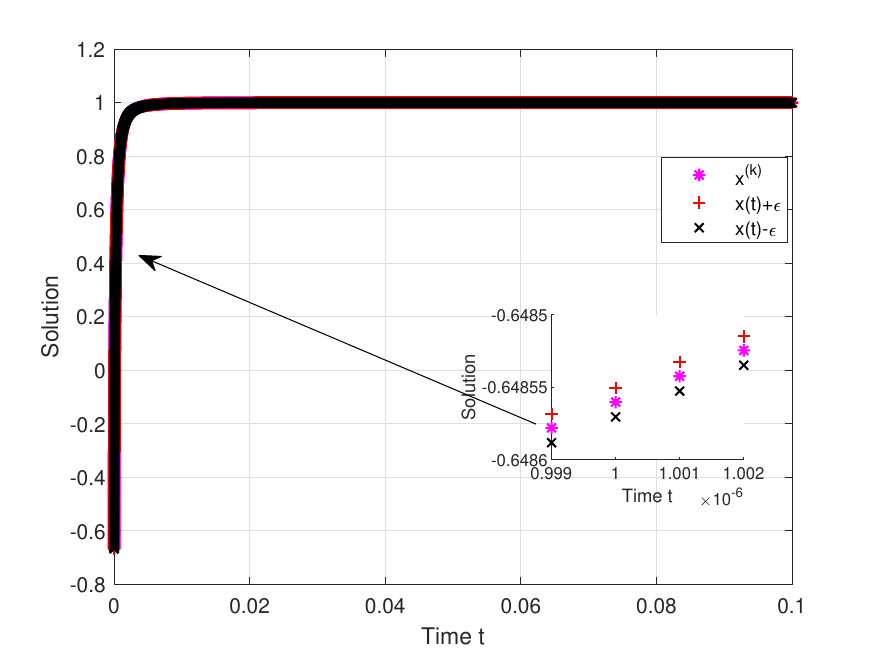}}
\end{tabular}\par
}\vspace{-0.15 cm}
\caption{The trajectory of $x(t)$ generated by the system \eqref{modelexam} with $t=\eta k$, and the sequence $\{x^{(k)}\}_{k=0}^{k_*}$ generated by the method \eqref{eq:discexam} with $x^{(0)}=x(0)$. For the left image, $x(0) = 2$ is used; for the middle image, $x(0) = 1 - \dfrac{1}{80\sqrt{2}}$; and for the right image, $x(0) = -\frac{2}{3}$.}
\label{figcompre}
\end{figure}
\end{example}

\begin{remark}{\rm
Theorem~{\rm\ref{thm:disc}} shows that the sequence $\{x^{(k)}\}_{k=0}^{\infty}$ generated by the method~\eqref{eq:disc} derived from the continuous dynamical system~\eqref{stem:model53} converges to an $\epsilon$-neighborhood of the unique solution~$x_*$ of GAVE~\eqref{eq:gave} within $k_*$-steps. When~$k_* = \lceil\frac{1}{\eta \gamma}\rceil$ is smaller, the sequence $\{x^{(k)}\}_{k=0}^{\infty}$ will converge to an $\epsilon$-neighborhood of~$x_*$ faster later on. Corollary~{\rm\ref{cor:xi}} indicates that the equilibrium point of the continuous dynamical system~\eqref{stem:model53} is globally stable at~$x_*$ within a fixed time $T_{\max} < \frac{1}{\gamma}$.  Hence, the convergence behavior of the sequence $\{x^{(k)}\}_{k=0}^{\infty}$ generated by the method~\eqref{eq:disc} may come close to mimicking the stability behavior of the equilibrium point of the continuous dynamical model~\eqref{stem:model53}.
}
\end{remark}

\section{Numerical experiments}\label{sec4:numerical}
In this section, two numerical examples are presented to demonstrate the validity and efficiency of the continuous dynamical system~\eqref{stem:model53} and its forward Euler discretization scheme~\eqref{eq:disc} for solving GAVE~\eqref{eq:gave}. All methods are implemented in MATLAB R$2021$a for Windows $11$ on a personal computer with Intel(R) Corel(TM) i$5$-$11320$H CPU @ $3.20$ GHz and $16$ GB memory.

Firstly, we test the performance of the dynamical system~\eqref{stem:model53} for solving a GAVE with $B$ being singular. The model~\eqref{eq:nn4gave} and the model in \cite{bsnc2019} (denoted it as model SM) with the smoothing functions $\phi_i$ and $i\in\{1,2,\cdots,8\}$ are compared. For dynamical models, the ODE solver ``ode45'' is used as follows:
\begin{center}
[$t$,$x$] = {\rm ode}45(odefun,tspan,$x_{0}$).
\end{center}

\begin{example}\label{exam1}{\rm
Consider GAVE \eqref{eq:gave} with
$$
A =\begin{bmatrix}
                                     8&-1&0&\cdots&0&0\\
                                     -1&8&-1&\cdots&0&0\\
                                     \vdots&\vdots&\vdots&\ddots&\vdots&\vdots\\
                                     0&0&0&\cdots&8&-1\\
                                     0&0&0&\cdots&-1&8\end{bmatrix},\quad B=I - \dfrac{vv^\top}{\|v\|^2},\quad v=[-\frac{1}{2},1,\cdots,-\frac{1}{2},1]^\top
$$
and $c=Ax_*-B|x_*|$, where $x_*=[\frac{1}{2},1,\frac{1}{2},1,\cdots,\frac{1}{2},1]^\top\in \mathbb{R}^n$.

For model~\eqref{stem:model53}, we set $\gamma=10$, $\xi=10$, $\rho_1=100$, $\rho_2=\frac{\xi^2\pi^2}{\rho_1\left(\frac{\sigma_{\min}(A)}{\|A\|}\right)^{\frac{1}{\xi}}(\sigma_{\min}(A) - \|B\|)^4}$, and $x_{0}=[0,0,0,0,\dots,0,0]^\top\in\mathbb{R}^{20}$, then~$T_{\max}=\frac{1}{\gamma}=0.1$. For the model~\eqref{eq:nn4gave}, $\rho=100$ and~$z_{0} = x_{0}$. For the model SM, $\rho_{\rm SM}=1$ and $\mu_0=0.3$. Numerical simulations are shown in Figure~{\rm\ref{fig5}}, from which we know that model~\eqref{stem:model53} converges faster than  model~\eqref{eq:nn4gave} and the model SM.

\begin{figure}[!htbp]\small
{\centering
\begin{tabular}{ccc}
\hspace{-0.3 cm}
\resizebox*{0.45\textwidth}{0.3\textheight}{\includegraphics{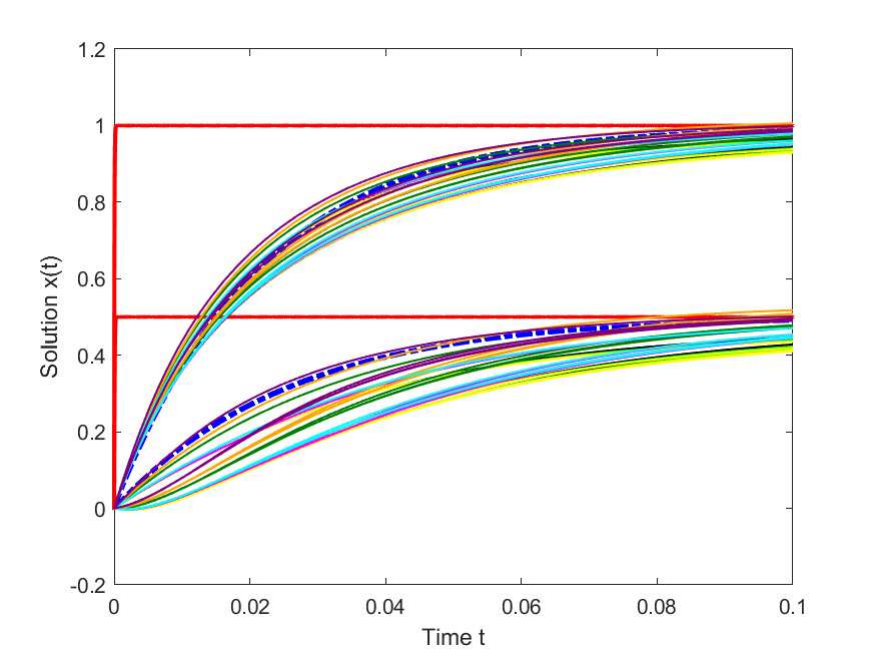}}
& \hspace{-0.9 cm}
\resizebox*{0.45\textwidth}{0.3\textheight}{\includegraphics{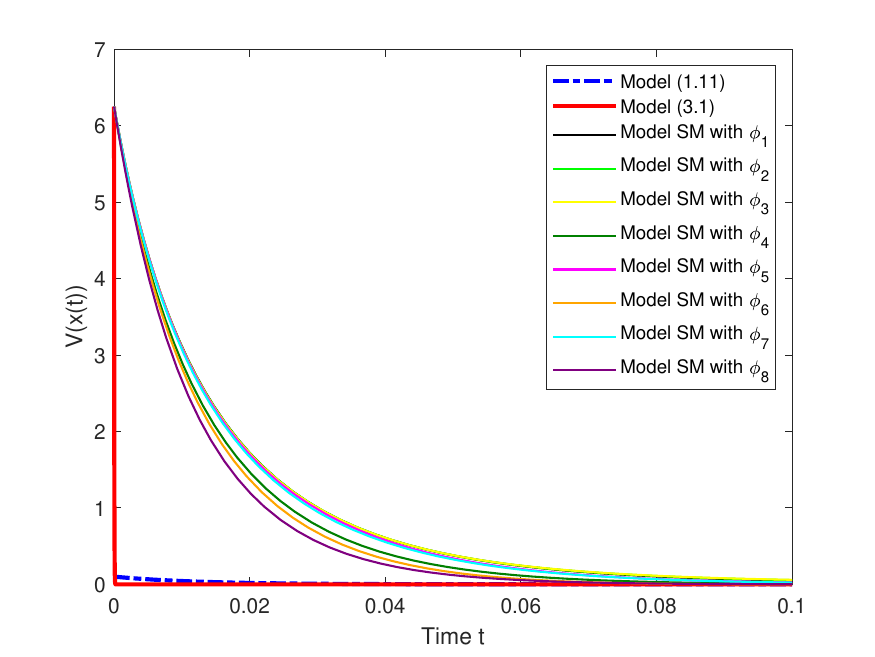}} \vspace{2ex}
\end{tabular}\par
}\vspace{-0.05 cm}
\caption{Phase diagrams for the model \eqref{stem:model53} (red solid line), the model \eqref{eq:nn4gave} (blue dotted line) and the model SM with the smoothing functions $\phi_i$ and $i\in\{1,2,\cdots,8\}$ respectively. Here, tspan = $[0 : T_{\max}]$.}
\label{fig5}
\end{figure}

}
\end{example}

In the following, we will test the performance of forward Euler discretization scheme~\eqref{eq:disc} (we set $\eta=10^{-8}$,~$\rho_1=1000,~\rho_2=\frac{\xi^2\pi^2}{\rho_1\left(\frac{\sigma_{\min}(A)}{\|A\|}\right)^{\frac{1}{\xi}}(\sigma_{\min}(A) - \|B\|)^4}$, $\xi=10$ and $\gamma=100$) for solving GAVE~\eqref{eq:gave}. The following nine methods are compared.

\begin{description}

  \item \textbf{GN}: the generalized Newton method \cite{mang2009a,huhz2011}
  \begin{equation*}
  x^{(k+1)}=\left[A-B\mathcal{D}(x^{(k)})\right]^{-1}c.
  \end{equation*}

  \item \textbf{Picard}: the Picard iteration method \cite{rohn2014}
  \begin{equation*}
    x^{(k+1)} = A^{-1} (B|x^{(k)}| + c).
  \end{equation*}

  \item \textbf{MN}: the modified Newton iteration method \cite{wacc2019}
  \begin{equation*}
    x^{(k+1)} = x^{(k)} - (A+\Omega)^{-1}(Ax^{(k)} - B|x^{(k)}| - c),
  \end{equation*}
where~$\Omega$ is a positive semi-definite matrix.

  \item \textbf{SSMN}: the shift-splitting MN iteration method \cite{liyi2021}
  \begin{equation*}
    x^{(k+1)} = x^{(k)} - 2(A+\Omega)^{-1}(Ax^{(k)} - B|x^{(k)}| - c),
  \end{equation*}
where~$\Omega$ is a positive semi-definite matrix.

  \item \textbf{FPI}: the fixed point iteration method \cite{keyf2020,lild2022}
 \begin{equation*}\label{eq:fpi4gave}
\begin{cases}
x^{(k+1)}=A^{-1}(B y^{(k)} + c),\\
y^{(k+1)}=(1-\omega_{\rm FPI})y^{(k)} + \omega_{\rm FPI}|x^{(k+1)}|,
\end{cases}
\end{equation*}
 where~$\omega_{\rm FPI}$ is a positive constant.
 \item \textbf{MFPI}: the modified fixed point iteration method \cite{lich2025b}:
 \begin{equation*}\label{eq:mfpi4gave}
\begin{cases}
x^{(k+1)}= A^{-1}(BQ_{\rm MFPI} y^{(k)} + c ),\\
y^{(k+1)}= (1-\omega_{\rm MFPI} )y^{(k)} + \omega_{\rm MFPI} Q^{-1}_{\rm MFPI}|x^{(k+1)}|,
\end{cases}
\end{equation*}
 where $\omega_{\rm MFPI}$ is a positive constant and $Q_{\rm MFPI}=10.5I$.
  \item \textbf{SOR}: the SOR-like iteration method \cite{kema2017,lich2025b}
\begin{equation*}\label{it:sor4gave}
		\begin{cases}
		    x^{(k+1)}=(1-\omega_{\rm SOR})x^{(k)} + \omega_{\rm SOR} A^{-1}\left( B y^{(k)} + c\right),\\
			y^{(k+1)}=(1-\omega_{\rm SOR})y^{(k)} + \omega_{\rm SOR} |x^{(k+1)}|,
		\end{cases}
\end{equation*}
where~$\omega_{\rm SOR}>0$.

\item \textbf{TS}: the two-step iteration method \cite{zzlf2022}:
\begin{equation*}\label{it:two-step4gave}
  x^{(k+1)} = A^{-1}(\omega_{\rm TS} x^{(k)} + B|x^{(k)}| - \omega_{\rm TS} x^{(k-1)} + c),
\end{equation*}
where $\omega_{\rm TS}$ is a given parameter and $x^{(1)} = x^{(0)}$.

  \item \textbf{NSNA}: the non-monotone smoothing Newton algorithm \cite[Algorithm~$1$]{cyhm2024} with the same parameters used in \cite{cyhm2024}.
\end{description}

For discrete iteration methods, all methods will be run ten times and the average IT (the number of iterations), the average CPU (the elapsed CPU time in seconds) and the average RRES are reported,  where
$$
{\rm RRES}:=\frac{\|c + B|x^{(k)}| - Ax^{(k)}\|}{\|c\|}.
$$
We set $x^{(0)}=[-1,0,-1,\cdots,-1,0]^\top\in\mathbb{R}^n$ and $y^{(0)}=c$ (if needed) for Example \ref{exam2}. If RRES $\leq10^{-8}$, then the tested methods are terminated.

\begin{example}[{\cite[Example 5.1]{lich2025}}]\label{exam2}{\rm
Let $A=\tilde{A}+\dfrac{1}{5}I$ and
$$B=
\left[
  \begin{array}{cccccccccccc}
    S_2 & -I & -I &-I &-I & 0 & 0 &0 & 0 & 0 & 0 & 0 \\
    -I & S_2 & -I &-I &-I & -I & 0 &0 & 0 & 0 & 0 & 0\\
    -I & -I & S_2 &-I &-I & -I & -I &0 & 0 & 0 & 0 & 0\\
    -I & -I & -I &S_2 &-I & -I & -I &-I & 0 & 0 & 0 & 0\\
    -I & -I & -I &-I &S_2 & -I & -I &-I & -I & 0 & 0 & 0\\
    0 & -I & -I &-I &-I & S_2 & -I &-I & -I & -I & 0 & 0\\
    \ddots&\ddots&\ddots&\ddots&\ddots&\ddots&\ddots
    &\ddots&\ddots&\ddots&\ddots&\ddots\\
    0 & 0&0&-I & -I &-I &-I & S_2 & -I &-I & -I & -I \\
    0 & 0&0&0&-I & -I &-I &-I & S_2 & -I &-I & -I  \\
    0 & 0&0&0&0&-I & -I &-I &-I & S_2 & -I &-I   \\
    0 & 0&0&0&0&0&-I & -I &-I &-I & S_2 & -I   \\
    0 & 0&0&0&0&0&0&-I & -I &-I &-I & S_2
  \end{array}
\right]
\in\mathbb{R}^{n\times n},$$
where	
$$
\tilde{A}=\scriptsize
\left[
  \begin{array}{cccccccccccc}
    S_1 & -1.5I & -0.5I &-1.5I &-0.5I & 0 & 0 &0 & 0 & 0 & 0 & 0 \\
    -1.5I & S_1 & -1.5I &-0.5I &-1.5I & -0.5I & 0 &0 & 0 & 0 & 0 & 0\\
    -0.5I & -1.5I & S_1 &-1.5I &-0.5I & -1.5I & -0.5I &0 & 0 & 0 & 0 & 0\\
    -1.5I & -0.5I & -1.5I &S_1 &-1.5I & -0.5I & -1.5I &-0.5I & 0 & 0 & 0 & 0\\
    -0.5I & -1.5I & -0.5I &-1.5I &S_1 & -1.5I & -0.5I &-1.5I & -0.5I & 0 & 0 & 0\\
    0 & -0.5I & -1.5I &-0.5I &-1.5I & S_1 & -1.5I &-0.5I & -1.5I & -0.5I & 0 & 0\\
    \ddots&\ddots&\ddots&\ddots&\ddots&\ddots&\ddots
    &\ddots&\ddots&\ddots&\ddots&\ddots\\
    0 & 0&0&-0.5I & -1.5I &-0.5I &-1.5I & S_1 & -1.5I &-0.5I & -1.5I & -0.5I \\
    0 & 0&0&0&-0.5I & -1.5I &-0.5I &-1.5I & S_1 & -1.5I &-0.5I & -1.5I  \\
    0 & 0&0&0&0&-0.5I & -1.5I &-0.5I &-1.5I & S_1 & -1.5I &-0.5I   \\
    0 & 0&0&0&0&0&-0.5I & -1.5I &-0.5I &-1.5I & S_1 & -1.5I   \\
    0 & 0&0&0&0&0&0&-0.5I & -1.5I &-0.5I &-1.5I & S_1
  \end{array}
\right],
$$
$$
S_1=
\left[
  \begin{array}{cccccccccc}
    36 & -1.5 & -0.5 &-1.5 & 0 &0 & 0 & 0 & 0 & 0 \\
    -1.5 & 36 & -1.5 &-0.5 &-1.5 &  0 &0 & 0 & 0 & 0 \\
    -0.5 & -1.5 & 36 &-1.5 &-0.5 & -1.5 & 0 & 0 & 0 & 0\\
    -1.5 & -0.5 & -1.5 &36 &-1.5 & -0.5 & -1.5 & 0 & 0 & 0\\
    0 & -1.5 & -0.5 &-1.5 &36 & -1.5 & -0.5 &-1.5 & 0 & 0\\
    0 & 0 & -1.5 &-0.5 &-1.5 & 36 & -1.5 &-0.5 & -1.5  & 0\\
    \ddots&\ddots&\ddots&\ddots&\ddots&\ddots&\ddots
    &\ddots&\ddots&\ddots\\
     0&0&0&0 & -1.5 &-0.5 &-1.5 & 36 & -1.5 &-0.5   \\
     0&0&0&0&0& -1.5 &-0.5 &-1.5 &36 & -1.5   \\
     0&0&0&0&0&0& -1.5 &-0.5 &-1.5 & 36
  \end{array}
\right]
\in\mathbb{R}^{m\times m},$$
$$S_2=
\left[
  \begin{array}{cccccccccc}
    3 & -1 & -1 &-1 & 0 &0 & 0 & 0 & 0 & 0 \\
    -1 & 3 & -1 &-1 &-1 &  0 &0 & 0 & 0 & 0 \\
    -1 & -1 & 3 &-1 &-1 & -1 & 0 & 0 & 0 & 0\\
    -1 & -1 & -1 &3 &-1 & -1 & -1 & 0 & 0 & 0\\
    0 & -1 & -1 &-1 &3 & -1 & -1 &-1 & 0 & 0\\
    0 & 0 & -1 &-1 &-1 & 3 & -1 &-1 & -1  & 0\\
    \ddots&\ddots&\ddots&\ddots&\ddots&\ddots&\ddots
    &\ddots&\ddots&\ddots\\
     0&0&0&0 & -1 &-1 &-1 & 3 & -1 &-1   \\
     0&0&0&0&0& -1 &-1 &-1 &3 & -1   \\
     0&0&0&0&0&0& -1 &-1 &-1 & 3
  \end{array}
\right]
\in\mathbb{R}^{m\times m},$$
and $c = Ax_*-B|x_*|$ with $x_*=[\dfrac{1}{2},1,\dfrac{1}{2},1,\dots,\dfrac{1}{2},1]^\top\in\mathbb{R}^{n}$ and $n = m^2$.

For the MN and SSMN methods, we consider two cases of~$\Omega$, i.e., $\Omega=1.5 D_A$ and $\Omega=2 D_A$, where $D_{A}$ is the diagonal part of $A$. For the SOR and FPI methods, $\omega_{\rm opt}$ is the experimentally found optimum value from $[0:0.01:2]$, where the MATLAB-like notation is used. Numerical results for Example~\ref{exam2} are reported in Table~\ref{table4}.

\setlength{\tabcolsep}{5.0pt}
\begin{table}[!htp]\scriptsize
\centering
\caption{Numerical results for Example~\ref{exam2}.}\label{table4}
\begin{tabular}{ccccccc}\hline
			Method  & $m$                            & $50$              & $60$              & $70$               & $80$              & $90$             \\ \hline
			\multirow{4}{*}{Method~\eqref{eq:disc}}&  $\eta=10^{-8}$  &                   &                    &                   &   &             \\
			        &  IT                            & $49$              & $50$              & $51$               & $52$              & $53$              \\
			        &  CPU                           & $\textbf{0.0320}$ & $\textbf{0.0361}$ & $\textbf{0.0404}$  & $\textbf{0.0677}$ & $\textbf{0.0790}$ \\
			        &  RRES                          & 3.0080e-09        & 5.8967e-09        & 6.5229e-09         & 5.2690e-09        & 3.3138e-09        \\\hline
			        \multirow{3}{*}{GN}&  IT         & $2$               &  $2$              & $2$                & $2$               & $2$               \\
			        &  CPU                           & $0.0519$          & $0.0589$          &  $0.0652$          &  $0.1114$         & $0.1314$         \\
			        & RRES                           & 6.0327e-16        & 6.4290e-16        & 6.5112e-16         & 7.0425e-16        & 7.2291e-16        \\\hline
			        \multirow{3}{*}{Picard}&  IT     & $26$              & $26$              & $26$               & $26$              &  $26$            \\
			        &   CPU                          &$0.3043$           & $0.3380$          &  $0.4000$          & $0.6714$          & $0.7996$         \\
			        &  RRES                          &5.9279e-09         & 6.9693e-09        & 7.7214e-09         & 8.2848e-09        & 8.7217e-09       \\\hline
			\multirow{8}{*}{MN}      & $\Omega = 2 D_A$     &                   &                   &                    &                   &     \\
			        &   IT                           &$47$               & $47$              & $47$               & $47$              &  $47$            \\
			        &   CPU                          &$0.5637$           & $0.6024$          &  $0.7503$          &  $1.4232$         & $1.5318$         \\
			        &  RRES                          &7.8392e-09         & 7.5124e-09        & 7.2750e-09         & 7.0945e-09        & 6.9526e-09       \\
              & $\Omega = 1.5 D_A$ &                   &                   &                    &                   &     \\
			        &   IT                           &$36$               & $36$              & $36$               & $36$              &  $36$            \\
			        &   CPU                          &$0.3009$           & $0.4518$          &  $0.5487$          &  $0.8025$         & $1.1731$         \\
			        &  RRES                          &9.0862e-09         & 8.5306e-09        & 8.1245e-09         & 7.8139e-09        & 7.5684e-09       \\\hline
       \multirow{8}{*}{SSMN}      & $\Omega= 2 D_A$                   &                   &                   &                    &                   & \\
			        &  IT                            & $18$              &  $18$             & $18$               & $18$              &  $18$            \\
			        &  CPU                           &$0.2482$           & $0.2391$          &  $0.2879$          &  $0.5043$         & $0.5802$          \\
			        & RRES                           &5.4771e-09         & 5.0798e-09        & 4.7858e-09         & 4.5585e-09        & 4.3772e-09       \\
              & $\Omega = 1.5 D_A$ &                   &                   &                    &                   &     \\
			        &   IT                           &$12$               & $12$              & $12$               & $12$              &  $12$            \\
			        &   CPU                          &$0.0968$           & $0.1561$          &  $0.1889$          &  $0.2708$         & $0.3827$         \\
			        &  RRES                          &8.0784e-09         & 7.2010e-09        & 6.5528e-09         & 6.0510e-09        & 5.6488e-09       \\\hline
           \multirow{4}{*}{FPI}      &  $\omega_{\rm opt}$    & $0.8$      & $0.8$             & $0.8$             & $0.8$            & $0.79$            \\
			        &  IT                            & $17$              & $17$              & $17$               & $17$              &  $17$            \\
			        &   CPU                          &$0.2262$           & $0.2075$          &  $0.2761$          &  $0.4285$         & $0.5709$         \\
			        &   RRES                         &9.7959e-09         & 9.2742e-09        & 8.8437e-09         & 8.4833e-09        & 9.7848e-09       \\\hline
           \multirow{4}{*}{MFPI}      &  $\omega_{\rm opt}$    & $0.79$      & $0.79$             & $0.79$             & $0.79$            & $0.79$            \\
			        &  IT                            & $22$              & $22$              & $22$               & $22$              &  $22$            \\
			        &   CPU                          &$0.2827$           & $0.2814$          &  $0.3616$          &  $0.5534$         & $0.7526$         \\
			        &   RRES                         &9.3122e-09         & 8.8271e-09        & 8.3838e-09         & 7.9891e-09        & 7.6394e-09       \\\hline
            \multirow{4}{*}{SOR}    & $\omega_{\rm opt}$                 & $0.9$            &$0.9$              &$0.9$             &$0.9$             &$0.9$   \\
			        &  IT                            &$16$               &  $16$             & $16$               & $16$              &  $16$            \\
			        &  CPU                           &$0.2185$           & $0.2025$          &  $0.2568$          &  $0.3862$         & $0.5223$          \\
			        &   RRES                         &9.0688e-09         & 8.4911e-09        & 7.9995e-09          & 7.5790e-09        & 7.2158e-09       \\\hline
            \multirow{4}{*}{TS}    & $\omega_{\rm opt}$                 & $0.8$            &$0.81$              &$0.79$             &$0.8$             &$0.8$   \\
			        &  IT                            &$15$               &  $15$             & $16$               & $16$              &  $16$            \\
			        &  CPU                           &$0.2656$           & $0.1902$          &  $0.2653$          &  $0.4077$         & $0.5169$          \\
			        &   RRES                         &7.6822e-09         & 9.5523e-09        & 9.5487e-09         & 6.9912e-09        & 7.9659e-09       \\\hline
			        \multirow{3}{*}{NSNA}&  IT       & $3$               &  $3$              & $3$                & $3$               & $3$              \\
			        &  CPU                           &$0.6152$           & $0.8677$          &  $1.2606$          &  $2.9163$         & $3.3705$          \\
			        & RRES                           &1.6817e-12         & 3.4639e-12        & 6.3858e-12         & 1.0853e-11        & 1.7333e-11       \\\hline
\end{tabular}
\end{table}

 As shown in Table \ref{table4}, all of the tested methods converge, among which iteration
 method~\eqref{eq:disc} requires the least CPU time. The main reason may be that it is inverse-free.

 }
\end{example}

\section{Conclusion}\label{sec5:conclusion}
A fixed-time inverse-free dynamic model for solving~{\rm GAVE}~\eqref{eq:gave} is presented. Under mild conditions, we proved that the unique equilibrium point of the proposed model is equivalent to the unique solution of~{\rm GAVE}~\eqref{eq:gave}. Theoretical results show that the proposed method globally converges to the unique solution of GAVE and has a conservative settling-time. For AVE~\eqref{eq:ave}, comparing with the existing fixed-time inverse-free dynamic model, the proposed method obtain a tighter upper bound of the settling-time. Furthermore, it is shown that the forward-Euler discretization of the proposed dynamic system results in an explicit $(T,\epsilon)$-close discrete-time approximation scheme. Numerical results demonstrate our claims.

\end{document}